\documentclass[12pt]{amsart}
\usepackage{amsmath,amssymb}

\newtheorem{theorem}{Theorem}[section]

\newtheorem{lemma}[theorem]{Lemma}
\newtheorem{proposition}[theorem]{Proposition}

\theoremstyle{definition}

\newtheorem{example}[theorem]{Example}

\theoremstyle{remark}

\newcommand{\C}{\mathbb C}
\def \bC{\mathbb C}
\def\bR{\mathbb R}
\def\im{\text{{\rm Im }}}
\def\re{\text{{\rm Re }}}

\DeclareMathOperator{\jac}{Jac}
\title[CR Transversality of holomorphic mappings]{CR transversality of holomorphic mappings between generic submanifolds in complex space}
\author{Peter Ebenfelt}
\address{Department of Mathematics, University of California at San Diego, La Jolla, CA 92093-0112}
\email{pebenfel@math.ucsd.edu}
\author{Duong Ngoc Son}
\address{Department of Mathematics, University of California at San Diego, La Jolla, CA 92093-0112}
\email{snduong@math.ucsd.edu}
\thanks{The first author was partly supporting by the NSF grant DMS-0701121. The second author acknowledges a scholarship from the Vietnam Education Foundation for his graduate study at UC San Diego and a grant from NAFOSTED (Vietnam).}
\begin{document}
\begin{abstract} We show that a holomorphic mapping sending one generic submanifold into another of the same dimension is CR transversal to the target submanifold provided that the source manifold is of finite type and the map is of generic full rank. This result and its corollaries completely resolve two questions posed by Linda P.~Rothschild and the first author in a paper from 2006.
\end{abstract}
\maketitle
\section{Introduction}
In this paper, we study CR transversality of holomorphic mappings sending a generic submanifold in $\C^N$ into another of the same dimension. Recall that if $U$ is an open subset of
$\C^{N}$, $H$ a holomorphic mapping $ U\to \C^{N}$, and
$M'$ a generic submanifold through a point $p':=H(p)$ for some $p\in U$, then $H$ is
said to be {\it CR transversal} to $M'$ at $p$ if
\begin{equation}
T^{1,0}_{p'}
M'+dH(T^{1,0}_p\C^N)=T^{1,0}_{p'}\C^N,
\end{equation}
where $T^{0,1}M' = \C TM' \cap T^{0,1}\C^N$ denotes the CR bundle on $M'$ and $T^{1,0}M' = \overline{T^{0,1}M'}$  its complex conjugate. We remark that CR transversality of a holomorphic mapping is a property that in general is strictly stronger than that of transversality as a smooth mapping, i.e.\ CR transversality implies transversality but the converse does not hold in general.
Transversality is an important and basic notion in geometry and analysis. For instance, an example of a transversality result in analysis is the  classical Hopf Lemma for subharmonic functions in a smoothly bounded domain in $\bR^n$ (see e.g.\ \cite{GT83}). The reader is referred to e.g.\ the paper \cite{ER06} for further elaboration on the significance of tranversality in geometry and analysis. Also, the reader is referred to Section \ref{prelim} for definitions and further explanation of the notions that appear in this introductory section.

The problem of CR transversality has been considered in various situations by many authors, see e.g.\ \cite{Forn76}, \cite{Forn78}, \cite{P77}, \cite{BB82}, \cite{BRgeom}, \cite{BRhopf},  \cite{CR94}, \cite{CR98}, \cite{BHR95},
\cite{HP96}, \cite{ER06}, \cite{BER07} and the references therein.  In this paper, we will consider the situation in which there is a generic submanifold $M\subset U$, of the same dimension as $M'$, such that $p\in M$ and $H(M)\subset M'$. In \cite{ER06}, this same situation is considered and it is proved that the map $H$ is CR transversal to $M'$ at $p\in M$ provided that $M'$ is of finite type (in the sense of Kohn and Bloom-Graham) at $p':=H(p)$ and the restriction of $H$ to the Segre variety of $M$ at $p$ is finite (as a mapping into the Segre variety of $M'$ at $p'$). The condition on $H$ is satisfied if, for instance, $H$ itself is a finite mapping $(\bC^N,p)\to (\bC^N, p')$ (i.e.\ the inverse image $H^{-1}(p')$ equals the singleton $\{p\}$ as a germ of a variety at $p$). By combining several results including the one mentioned above, it is noted in \cite{ER06} that if $M$ and $M'$ are {\it hypersurfaces} (i.e.\ submanifolds of real codimension one) and $M$ is essentially finite at $p$, a stronger condition than finite type in the hypersurface case, then $H$ is CR transversal to $M'$ at $p'$ under the assumption that $\jac H\not\equiv 0$, where $\jac H$ denotes the Jacobian determinant of $H$ and $\not\equiv$ means ``not identically zero"; we remark that $\jac H\not\equiv 0$ is strictly weaker than being a finite map. It was conjectured in \cite{ER06} that the condition of essential finiteness of $M$ at $p$ could be weakened to finite type at $p$ while maintaining the weak requirement $\jac H\not\equiv 0$ for the mapping. It was also conjectured that the condition $\jac H\not \equiv 0$ would be sufficient to conclude CR tranversality in the case of higher codimensional generic submanifolds $M$ and $M'$ provided that $M$ is assumed to be of finite type and essentially finite at $p$. (For higher codimensional generic submanifolds, the conditions of being of finite type and essentially finite are unrelated, and the authors of \cite{ER06} imposed both conditions on $M$ in this conjecture to be on the safe side.) In this paper, we prove both of these conjectures. Indeed, we even prove that the condition of essential finiteness in the higher codimensional case is superfluous. More precisely,  we have the following theorem.

\begin{theorem}\label{main} Let $M, M' \subset \C^N$ be smooth generic submanifolds of the same dimension through $p$ and $p'$ respectively, and $H: (\C^N, p) \to (\C^N,p')$ a germ of a holomorphic mapping such that $H(M) \subset M'$. Assume that $M$ is of finite type at $p$ and $\jac H \not\equiv 0$. Then $H$ is CR transversal to $M'$ at $p'$.
\end{theorem}

We note that if $M'$ is of finite type at $p'$ and the restriction of $H$ to the Segre variety of $M$ at $p$ is finite, then $M$ is of finite type at $p$ and $\jac H\not\equiv 0$ (see Proposition 2.3 in \cite{ER06}). Thus, Theorem \ref{main} also implies the result in \cite{ER06} described above (Theorem 1.4 in \cite{ER06}).

The condition that $M$ is of finite type at $p$ in Theorem \ref{main} is optimal, as the following example (see Example 6.2 \cite{ER06}) shows.

\begin{example} {\rm Let $H:(\C^2,0)\to(\C^2,0)$ be the finite  mapping $H(z,w)=(z,w^2)$, and $M'\subset\C^2$ the hypersurface given  by $\im w=2(\re w)|z|^2$. If $M\subset \C^2$ denotes the hypersurface given by $\im w= (\re w)\xi(|z|^2)$, where $\xi=\xi(|z|^2)$ is the solution (whose existence and uniqueness is guaranteed by the Implicit Function Theorem) to the equation
\begin{equation}\label{xi}
\xi-(1-\xi^2)|z|^2=0,\qquad \xi(0)=0,
\end{equation}
then $H$ sends $M$ into $M'$, satisfies $\jac H\not\equiv0$, but is not CR transversal to $M'$ at $0$. To see that $H$ sends  $M$ into $M'$, let us denote by $\rho'(z,w,\bar z, \bar w):=\im w-2(\re w)|z|^2$ and note that $H$ sends $M$ into $M'$ if and only if $\rho'\circ H$ vanishes on $M$. In other words, if we write $w=s+it$,  then $H$ sends $M$ into $M'$ if and only if
$$
2st-2(s^2-t^2)|z|^2=0
$$
when $t=s\xi(|z|^2)$. The latter follows immediately from the definition \eqref{xi} of $\xi$.
}
\end{example}    

It is also clear that some {\it a priori} condition on the mapping $H$  is necessary for the conclusion of Theorem \ref{main} to hold, as is easily verified by simple examples. The reader is referred to e.g.\ Example 2.4 in \cite{BER07} (in which one may consider $M\subset\C^2$ as a finite type hypersurface in $\C^3$ via a standard embedding) for an example of a mapping $H$ (with $\jac H\equiv0$) that is not CR transversal to $M'$ at 0. If the source manifold $M$ is holomorphically nondegenerate, then in fact the condition $\jac H\not\equiv 0$ is necessary for CR transversality to hold (see Proposition \ref{ngocanh} below). We state here the following result, which is a direct consequence of Theorem \ref{main} and Proposition \ref{ngocanh}.

\begin{theorem}\label{holomorphicnondegeneracycase}
Let $M, M' \subset \C^N$ be smooth generic submanifolds of the same dimension through $p$ and $p'$ respectively, and $H: (\C^N, p) \to (\C^N,p')$ a germ of a holomorphic mapping such that $H(M) \subset M'$. Assume that $M$ is holomorphically nondegenerate and of finite type at $p$.  Then $H$ is CR transversal to $M'$ at $p$ if and only if $\jac H\not\equiv 0$.
\end{theorem}

If we impose stronger conditions on $M$, such as essential finiteness or finite nondegeneracy in addition to finite type, then we also obtain stronger conclusions as in \cite{ER06}. Indeed, by combining our Theorem \ref{main} with Theorems 6.1 and 6.6 in \cite{ER06}, we obtain the following result; the notions of of essential finiteness and finite nondegeneracy will not play a role in any of the proofs in this paper and, hence, the reader is referred to \cite{ER06} or the book \cite{BER99a} for their definitions.

\begin{theorem}\label{cor}
 Let $M, M' \subset \C^N$ be smooth generic submanifolds through $p$ and $p'$ respectively, and $H: (\C^N, p) \to (\C^N,p')$ a germ of a holomorphic mapping such that $H(M) \subset M'$ and $\jac H \not\equiv 0$ . If $M$ is of finite type and essentially finite at $p$, then $H$ is a finite map $(\C^N, p) \to (\C^N,p')$. If, in addition, $M$ is finitely nondegenerate at $p$, then $H$ is a local biholomorphism near $p$.
\end{theorem}

We mention that the notion of CR transversality can be defined also for CR mappings $f\colon M \to M'$ in terms of a formal power series expansion for $f$ (see e.g.\ \cite{ER06}). The analog of Theorem \ref{main} holds for CR mappings whose differentials have generic maximal rank.

In the next section, we will briefly recall some basic definitions and facts. In Section 3, we will present the proof of Theorem~\ref{main}. For further comments, examples and related results regarding this problem, we refer the reader to e.g.\ the  papers \cite{BRgeom} and \cite{ER06}.

\section{Preliminaries}\label{prelim}

In this section, we will recall some basic definitions and facts about real submanifolds in complex spaces. For more detail and proofs of facts stated, we refer the reader to the book \cite{BER99a}. Recall that a real submanifold $M$ of codimension $d$ in $\C^N$ ($\cong\bR^{2N}$) is said to be generic if, for every $p\in M$, the submanifold $M$ is defined locally near $p$ by a defining equation $\rho(Z,\bar Z) = 0$, where $\rho = (\rho_1,\dots \rho_d)$ is a smooth $\bR^d$-valued function satisfying the following condition
\[
\partial\rho_1\wedge\ldots \wedge \partial\rho_d\neq 0.
\]
In particular, a generic submanifold $M$ of codimension $d$ in $\bC^N$ is a CR manifold of CR dimension $n=N-d$. If $M$ is real-analytic and $p\in M$, then there are normal (local) coordinates $Z=(z,w)$, where $z=(z_1,\dots, z_n)$ and $w=(w_1,\dots w_d)$, vanishing at $p$ such that $M$ is defined by the complex equation
\begin{equation}\label{definingequation}
w = Q(z,\bar z,\bar w),
\end{equation}
where  $Q(z,\chi,\tau)$ is a $\C^d$-valued holomorphic function defined in a neighborhood of $p=(0,0)$ satisfying
\begin{equation}\label{normalcondition}
Q(z, 0 ,\tau) \equiv Q(0, \chi, \tau) \equiv \tau.
\end{equation}
The fact that the $d$ complex equations in \eqref{definingequation} define a  submanifold of real codimension $d$ is equivalent to the identity
\begin{equation}\label{reality}
Q(z,\chi, \bar Q(\chi, z ,w)) \equiv w;
\end{equation}
here and in what follows, we use the following notation: if $u(x)$ is an analytic function or formal power series in some set of variables $x$, then $\bar u(x)$ is the function or power series given by $\bar u(x):=\overline{u(\bar x)}$.

A submanifold $M$ is said to be of {\it finite type} (in the sense of Kohn and Bloom-Graham)  at
$p$ if the
(complex) Lie algebra
$\frak g_M$ generated by all smooth CR (or (0,1)) vector fields on $M$ and their conjugates satisfies $\frak
g_M(p)=\C T_{p}M$. Our proof of Theorem \ref{main} will rely partly on a characterization of finite type due to Baouendi, Rothschild and the first author given in \cite{BER96} (see also \cite{BER03}) in term of the generic rank of the iterated Segre mappings; precise details will be given in Section 3. A generic, real-analytic submanifold $M$ is said to be {\it holomorphically nondegenerate} at $p$ if there is no (nontrivial) germ at $p$ of a holomorphic vector field tangent to $M$ near $p$. By a germ at $p$ of a holomorphic vector field shall mean a vector field of the form
\[
L= \sum_{j=1}^N \varphi_j(Z) \frac{\partial }{\partial Z_j}
\]
where the $\varphi_j$ are germs at $p$ of holomorphic functions.

We shall consider the variables $\bar z, \bar w$ in \eqref{definingequation} as an independent set of complex variables $\xi = (\chi,\tau)$ and thus the complexified equation $w = Q(z,\chi,\tau)$ defines a complex submanifold $\mathcal M$ of codimension $d$ in $\C^N_{Z} \times \C^N_{\xi}$. We shall refer to $\mathcal M$ as the complexification of $M$. If $M, M'\subset \bC^N$ are generic, real analytic submanifolds of codimension $d$, given in normal coordinates by $w=Q(z,\bar z, \bar w)$ and $w' = Q'(z',\bar z',\bar w')$, and $H(z,w) = (F(z,w), G(z,w))$ is a holomorphic mapping $(\bC^N,0)\to (\bC^N,0)$ sending $M$ into $M'$, then there is a $d\times d$ matrix $a(z,w,\chi,\tau)$ of holomorphic functions in a neighborhood of the origin in $\bC^N_Z\times\bC^N_\xi$ such that the following identity holds:
\begin{equation}\label{mapping}
G(z,w) - Q'(F(z,w), \bar F(\chi,\tau), \bar G(\chi,\tau)) = a(z,w,\chi,\tau)(w-Q(z,\chi,\tau)).
\end{equation}
Equivalently, the complexified mapping $\mathcal H(z,w,\chi,\tau):=(H(z,w),\bar H(\chi,\tau))$ sends  $\mathcal M$ into $\mathcal M'$ (with the obvious notation). It is straightforward to verify that $H$ is CR transversal to $M'$ at $0$ if and only if $\det a(0)\neq 0$ (cf.\ \cite{BER07}).

It is convenient to relax the convergence properties of the map and defining equations. A formal, generic submanifold $M$ of codimension $d$ through $0$ in $\bC^N$ is defined by a formal equation of the form \eqref{definingequation}, where $Q(z,\chi,\tau)$ is a $\bC^d$-valued power series in $(z,\chi,\tau)\in \bC^n\times\bC^n\times\bC^d$ satisfying the normality condition \eqref{normalcondition} and the reality condition \eqref{reality}. A formal holomorphic mapping $H=(F,G)\colon (\bC^N,0)\to (\bC^N,0)$ (i.e.\ a $\bC^N$-valued power series in $Z=(z,w)$ with no constant term) is said to send the formal submanifold $M$ into a formal submanifold $M'$ if there exists a $d\times d$ matrix $a(z,w,\chi,\tau)$ of formal power series such that \eqref{mapping} holds. If $M$ and $M'$ are smooth, generic submanifolds through $p$ and $p'$ in $\bC^N$ and $H$ is a holomorphic mapping $(\bC^N,p)\to (\bC^N,p')$ (or a smooth CR mapping defined on $M$) sending $M$ into $M'$, then one can associate to $M$ and $M'$ formal manifolds, still denoted by $M$ and $M'$, through $0$ and a formal holomorphic mapping, also denoted by $H\colon (\bC^N,0)\to (\bC^N,0)$, sending $M$ into $M'$; the reader is e.g.\ referred to \cite{BER99b} for this (fairly obvious) construction. It is also straightforward to verify that the holomorphic mapping $H$ sending the smooth manifold $M$ into the smooth manifold $M'$ is CR transversal to $M'$ at $p$ if and only if the formal manifolds and mapping satisfy \eqref{mapping} with $\det a(0)\neq0$. Thus, to prove Theorem \ref{main} it suffices to consider formal manifolds $M$ and $M'$ through $0$ in $\bC^N$, a formal holomorphic mapping $H\colon (\bC^N,0)\to (\bC^N,0)$ sending $M$ into $M'$, i.e.\ satisfying \eqref{mapping}, and prove that the matrix $a(z,w,\chi,\tau)$ in \eqref{mapping} satisfies $\det a(0)\neq0$; the reader is also referred to \cite{ER06} for more details on this reduction to the formal case.
In what  follows, we shall consider only formal mappings and formal generic submanifolds.

\section{Proof of the main theorem and corollaries}
In this section, $M$ and $M'$ are generic formal submanifolds through 0, given in normal coordinates by $w=Q(z,\bar z, \bar w)$ and $w' = Q'(z',\bar z',\bar w')$,  and $H(z,w) = (F(z,w), G(z,w))$ a formal mapping sending $M$ into $M'$. We may consider $H$ and $\bar H$ as formal maps from $\C^N_Z\times C^N_\xi$ in the obvious way:
\begin{align*}
&H(z,w,\chi,\tau) = H(z,w)\\
&\bar H(z,w,\chi,\tau) = \bar H(\chi,\tau)
\end{align*}
We begin with the following lemma.

\begin{lemma}\label{matrixidentity} Let $M$ and $M'$ be generic formal submanifolds given in normal coordinates by $w=Q(z,\bar z, \bar w)$ and $w' = Q'(z',\bar z',\bar w')$ and $H=(F,G):(\C^N, 0 )\to (\C^N, 0)$ a formal mapping sending $M$ into $M'$, that is
\begin{equation}\label{mainequation}
G(z,w) - Q'(F(z,w), \bar F(\chi,\tau), \bar G(\chi,\tau)) = a(z,w,\chi,\tau)(w-Q(z,\chi,\tau)),
\end{equation}
where $a(z,w,\chi,\tau)$ is a $d\times d$ matrix of formal power series in $\C[[z,w,\chi,\tau]]$. Then, there are $d\times d$ matrices $C(z,\chi,\tau)$ and $E(z,\chi,\tau)$ of formal power series in $\C[[z,\chi,\tau]]$ such that the following holds
\begin{align}\label{a0}
\det H_Z(z,Q(z,\chi,\tau)) I_d &= a(z,Q(z,\chi,\tau),\chi,\tau) C(z,\chi,\tau)\\
\det \bar H_\xi(\chi,\tau)  I_d &= a(z,Q(z,\chi,\tau),\chi,\tau) E(z,\chi,\tau) \label{b0}
\end{align}
where $I_d$ denotes the $d\times d$ identity matrix.
\end{lemma}

\begin{proof} Differentiating \eqref{mainequation} with respect to $w$ and substituting $w=Q(z,\chi,\tau)$ we get
\begin{multline}\label{eq2}
G_w(z,Q(z,\chi,\tau))-  Q'_{z'}(F(z,Q(z,\chi,\tau),\bar F(\chi,\tau),\bar G(\chi,\tau))
 F_w(z,Q(z,\chi,\tau)) =\\ a(z,Q(z,\chi,\tau),\chi,\tau).
\end{multline}
To make our formulas more compact, we will use the notation $Z=(z,w)$ and $\xi = (\chi,\tau)$. We shall denote by $\Psi$ a parametrization of $\mathcal M$, i.e.\
\[
\Psi (z,\chi,\tau) = (z,Q(z,\chi,\tau), \chi,\tau).
\]
We shall also use the notation
\[
\Phi(z,\chi,\tau) = (F(z,Q(z,\chi,\tau)), \bar F(\chi,\tau), \bar G(\chi,\tau)).
\]
Thus, \eqref{eq2} can be written as
\begin{equation}\label{eq3}
G_w(\Psi) - Q'_{z'}(\Phi) F_w(\Psi) = a(\Psi).
\end{equation}
Similarly, we may differentiate \eqref{mainequation} with respect to $z$ and substitute $w=Q(z,\chi,\tau)$ to get
\begin{multline}
G_z(z,Q(z,\chi,\tau)) -  Q'_{z'}(F(z,Q(z,\chi,\tau),\bar F(\chi,\tau),\bar G(\chi,\tau))  F_z(z,Q(z,\chi,\tau))\\ = a(z,Q(z,\chi,\tau),\chi,\tau)(-Q_z(z,\chi,\tau)),
\end{multline}
or, in the notation above,
\begin{equation}\label{eq4}
G_z(\Psi) - Q'_{z'}(\Phi ) F_z(\Psi) = a(\Psi)\cdot (- Q_z).
\end{equation}
The equations \eqref{eq3} and \eqref{eq4} can be combined into the following equation
\begin{equation}\label{eq5}
(-Q'_{z'}(\Phi), I_d)
 \begin{bmatrix} F_z (\Psi) & F_w(\Psi) \\
 			G_z(\Psi) & G_w(\Psi)
\end{bmatrix} = a(\Psi)  ( - Q_z, I_d).
\end{equation}
Let $V':=( -Q'_{z'}(\Phi), I_d)$, $V := (- Q_z, I_d)$, and let $H_Z$ be the Jacobian matrix of $H$. Then the equation \eqref{eq5} can be written as follows:
\begin{equation}\label{eq6}
V'H_Z(\Psi) = a(\Psi) V.
\end{equation}
By Cramer's rule, there is an $N\times N$ matrix of formal power series $B(z,w)$  such that $H_Z B = BH_Z = (\det H_Z) I_N$. Thus, it follows from equation \eqref{eq6} that
\begin{equation}\label{eq7}
\det H_Z(\Psi)V' = a(\Psi) VB(\Psi).
\end{equation}
By considering the last $d$ columns of equation \eqref{eq7}, we find that there is a $d\times d$ matrix $C = C(z,\chi,\tau)$ whose entries are power series such that
\begin{equation*}
 \det H_Z(\Psi) I_d  =a(\Psi)C.
\end{equation*}
Thus,  equation \eqref{a0} is proved.

Now, to prove \eqref{b0}, we differentiate \eqref{mainequation} with respect to $\chi$ and substitute $w=Q(z,\chi,\tau)$:
\begin{multline}
Q'_{\chi'}(F(z,Q(z,\chi,\tau),\bar F(\chi,\tau),\bar G(\chi,\tau))  \bar F_\chi(\chi,\tau)
+ \\Q'_{\tau'}(F(z,Q(z,\chi,\tau),\bar F(\chi,\tau),\bar G(\chi,\tau)) \bar G_\chi(\chi,\tau)
= a(z,Q(z,\chi,\tau),\chi,\tau)  Q_\chi(z,\chi,\tau).
\end{multline}
In other words,
\begin{equation}\label{eq8}
Q'_{\chi'}(\Phi) \bar F_\chi (\Psi)+ Q'_{\tau'}(\Phi) \bar G_\chi (\Psi) = a(\Psi) Q_\chi.
\end{equation}
Similarly, we differentiate \eqref{mainequation} with respect to $\tau$ and substitute $w=Q(z,\chi,\tau)$:
\begin{equation}\label{eq9}
Q'_{\chi'}(\Phi) \bar F_\tau (\Psi) + Q'_{\tau'}(\Phi) \bar G_\tau(\Psi) = a(\Psi) Q_\tau.
\end{equation}
The equations \eqref{eq8} and \eqref{eq9} can be combined into the following equation
\begin{equation}\label{eq9a}
(Q'_{\chi'}(\Phi ), Q'_{\tau'}(\Phi ))
 \begin{bmatrix} \bar F_\chi (\Psi) & \bar F_\tau(\Psi) \\
 			\bar G_\chi(\Psi) & \bar G_\tau(\Psi)
\end{bmatrix} = a(\Psi) (Q_\chi, Q_\tau).
\end{equation}
Let $W = (Q_\chi, Q_\tau)$ and $W' = (Q'_{\chi'}(\Phi ), Q'_{\tau'}(\Phi ))$. It follows from \eqref{eq8} and \eqref{eq9} that
\begin{equation}\label{eq10}
W' \bar H_\xi (\Psi)  = a(\Psi) W.
\end{equation}
Note that $\bar H_\xi(\chi,\tau) \bar B(\chi,\tau) = \det \bar H_\xi(\chi,\tau)I_N  $, where $B(z,w)$ is the matrix introduced above. Multiplying both sides of \eqref{eq10} with $\bar B(\Psi)$, we obtain
\begin{equation}\label{eq11}
W' \det \bar H_\xi(\Psi) = a(\Psi) W\bar B(\Psi).
\end{equation}
Taking the last $d$ columns of \eqref{eq11}, we obtain
\begin{equation}\label{eq12}
\det \bar H_\xi(\Psi) Q'_{\tau'}(\Phi)= a(\Psi) D,
\end{equation}
where $D=D(z,\chi,\tau)$ is the matrix formed by the last $d$ columns of $W\bar B(\Psi)$.
Now, it follows from \eqref{normalcondition} that $Q'_{\tau'}(0,0,0) = I_d$ and hence  $Q'_{\tau'}(\Phi(z,\chi,\tau))$ is invertible over the ring $\C[[z,\chi,\tau]]$. Consequently, it follows from \eqref{eq12} that there is a matrix $E=E(z,\chi,\tau)$ such that
\begin{equation*}
\det \bar H_\xi(\Psi) I_d = a(\Psi) E,
\end{equation*}
which is \eqref{b0}.
The proof is complete.
\end{proof}

\begin{lemma}\label{divisor} Assume that $\det H_Z(0)=0$, but $\det H_Z(z,w) \not\equiv 0$. Then there exist units $u(z,w,\chi)$, $v(z,\chi,\tau)$ in $\C[[z,w,\chi,\tau]]$ and formal power series $b(z,w)\in \C[[z,w]]$, $c(\chi,\tau) \in \C[[\chi,\tau]]$ such that
\begin{align}
\label{div1}\det a(z,w,\chi,\bar Q(\chi, z,w)) &= u(z,w,\chi) b(z,w).\\ \label{div2}
\det a(z,Q(z,\chi,\tau),\chi,\tau)) &= v(z,\chi,\tau) c(\chi,\tau).
\end{align}
Furthermore, $b(z,w)$ is a divisor of $(\det H_z(z,w))^d$ in $\C[[z,w]]$ and $c(\chi,\tau)$ is a divisor of $(\det \bar H_\xi(\chi,\tau))^d$ in the ring $\C[[\chi,\tau]]$.
\end{lemma}

\begin{proof} It follows from \eqref{a0} that
\begin{equation*}
(\det H_Z(\Psi))^d = \det a(\Psi)\det C.
\end{equation*}
If we substitute $\tau= \bar Q(\chi, z,w)$ and use the identity \eqref{reality}, we obtain
\begin{equation}\label{eq13}
(\det H_Z(z,w))^d = \det a(z,w,\chi,\bar Q(\chi, z, w))\det C(z,\chi,\bar Q(\chi, z, w)).
\end{equation}
We now factor both sides of \eqref{eq13} into products of irreducible elements in the unique factorization domain $\C[[z,w,\chi]]$. Since the left hand side of \eqref{eq13} is a non-trivial formal power series in the ring $\C[[z,w]]\subset\C[[z,w,\chi]]$, its factorization involves factors that are power series in $z$ and $w$ only. Thus, by the uniqueness of the factorization, we obtain
\begin{equation*}
\det a(z,w,\chi,\bar Q(\chi, z, w)) = u(z,w,\chi) b(z,w).
\end{equation*}
where $b(z,w) \in \C[[z,w]]$ is a divisor of $(\det H_Z(z,w))^d$ and $u(z,\chi,\tau)$ is a unit in $\C[[z,w,\chi]]$.

Similarly, it follows from \eqref{b0} that
\begin{equation*}
(\det \bar H_\xi(\chi,\tau))^d= \det a(z,Q(z,\chi,\tau),\chi,\tau))  \det E(z,\chi,\tau).
\end{equation*}
A similar argument to the one above shows that
\begin{equation*}
\det a(z,Q(z,\chi,\tau),\chi,\tau)) = v(z,\chi,\tau) c(\chi,\tau),
\end{equation*}
where $v(z,\chi,\tau) \in \C[[z,\chi,\tau]]$ is a unit and $c(\chi,\tau) \in \C[[\chi,\tau]]$ is a divisor of $(\det \bar H_\xi(\chi,\tau))^d$. The proof is complete.
\end{proof}

\begin{lemma}\label{recursion} Let $u,v,b,c$ be as in Lemma~$\ref{divisor}$. Then, there is a unit $s(z,\chi,\tau)$ in $\C[[z,\chi,\tau]]$ such that
\begin{equation}\label{eq16}
b(z,Q(z,\chi,\tau)) = s(z,\chi,\tau) c(\chi,\tau).
\end{equation}
\end{lemma}

\begin{proof}
We substitute $w = Q(z,\chi, \tau)$ into \eqref{div1}, and use \eqref{reality} and \eqref{div2} to obtain
\begin{align}
u(z,Q(z,\chi,\tau),\chi)\cdot b(z,Q(z,\chi,\tau)) \notag
& = \det a(z,Q(z,\chi,\tau),\chi,\tau))\\
&= v(z,\chi,\tau)  c(\chi,\tau).
\end{align}
Since $u$ and $v$ are units, we can take $s(z,\chi,\tau) = (u(z,Q(z,\chi,\tau),\chi))^{-1}v(z,\chi,\tau)$ to obtain \eqref{eq16}. It is obvious that $s$ is also a unit. The proof is complete.
\end{proof}
We will need the following lemma whose proof may be found, e.g., in \cite{BER99a}, Proposition 5.2.3.
\begin{lemma}\label{genericrank}
Let $K:(\C^N, 0 )\to (\C^k,0)$ and $\phi : (\C^m, 0 )\to (\C^N,0)$ be formal mappings. Assume that $\phi$ has generic rank $N$. If $K \circ \phi \equiv 0$ then $K\equiv 0$.
\end{lemma}

We shall now prove our main result.
\begin{proof}[Proof of Theorem~$\ref{main}$]
As explained in Section 2, we may assume that $M$ and $M'$ are formal manifolds through $0\in\C^N$, $H=(F,G)\colon (\C^N,0)\to(\C^N,0)$ a formal mapping satisfying \eqref{mapping}, and then to prove Theorem \ref{main}, it suffices to  prove that the matrix $a(z,w,\chi,\tau)$ in \eqref{mapping} satisfies $\det a(0)\neq0$. Thus, we assume, in order to reach a contradiction, that $\det a(0)=0$.
We deduce from \eqref{div2} that
\begin{equation*}
v(0,0,0)\cdot c(0,0) = \det a(0) = 0.
\end{equation*}
Thus, $c(0,0) =0 $ since $v(z,\chi,\tau)$ is a unit. Setting $\chi =0 $ and $\tau=0$ in equation \eqref{eq16} yields
\begin{align}\label{eq16a}
 b(z,0) = s(z,0,0)\cdot c(0,0) = 0.
\end{align}
Thus, it follows from \eqref{div1} that
\begin{equation}\label{first}
\det a(z,0,\chi, \bar Q(\chi, z, 0))= u(z, 0, \chi)\cdot b(z,0) = 0.
\end{equation}
By taking determinants on both sides of equation \eqref{b0}, substituting $\tau=\bar Q(\chi,z,0)$, and using \eqref{reality}, we conclude that
\begin{align}\label{eq17}
(\det \bar H_\xi&(\chi, \bar Q(\chi, z, 0)))^d \\ \notag
&= \det a(z,0,\chi, \bar Q(\chi, z, 0)) \det E(z, \chi, \bar Q(\chi, z, 0)) \equiv 0.
\end{align}
 In the hypersurfaces case, the proof is complete. Indeed,  if $M$ is finite type at $0$, then the map $(z,\chi)\mapsto (\chi, \bar Q(\chi, z, 0))$ has rank $n+1$ (see e.g \cite{BER96}) and thus, by Lemma~\ref{genericrank}, equation \eqref{eq17} implies that $\det \bar H_\xi \equiv 0$. This is a contradiction.

For general case, we will need the iterated Segre mappings introduced in \cite{BER96} (see also \cite{BER03}, \cite{ER06}). For a positive integer $k$, the {\it $k$th Segre mapping} of $M$ at $0$ is the
mapping $v^k: \C^{kn}\to\C^N$ defined by
\begin{equation}\label{e:k-segre}
\C^{kn}\ni t=(t^1,\ldots,t^k)\mapsto v^k(t):=(t^k,u^k(t^1,\ldots,t^k)),
\end{equation}
where $u^k:\C^{kn}\to \C^d$ is given inductively by
\begin{equation}\label{e:uk}
u^1(t^1)=0, \ u^k(t^1,\ldots,t^k)=Q(t^k,
t^{k-1},\overline{ u^{k-1}}(t^1,\ldots,t^{k-1})),\ k
\ge 2.
\end{equation}
The crucial property of the Segre mappings needed here is the result (see \cite{BER96}, \cite{BER99a}, \cite{BER99b}, \cite{BER03}) that $M$ is of finite type at $0$ if and only if the maps $v^{k}$ has generic rank $N$ for  $k$ large enough ($k\geq d+1$ suffices). Thus, by Lemma~\ref{genericrank}, the following lemma implies that $\det H_Z \equiv 0$, which is a contradiction and completes the proof of Theorem \ref{main}.
 \end{proof}

\begin{lemma}\label{propagation}
For every $j\ge 0$, the following holds.
\begin{equation}\label{e:propagation}
\det H_Z \circ v^{2j+1} \equiv 0.
 \end{equation}
\end{lemma}

 \begin{proof} We may consider $b(z,w)$ and $c(\chi,\tau)$ in Lemmas \ref{divisor} and \ref{recursion} as power series in $(z,w,\chi,\tau)$ by  $b(z,w,\chi,\tau) =b(z,w)$ and $c(z,w,\chi,\tau) = c(\chi,\tau)$. Since the complexification $\mathcal M$ of $M$ is parametrized by
 $$(z,\chi,\tau) \mapsto (z,Q(z,\chi,\tau),\chi,\tau),$$
  it follows from \eqref{eq16} that $b \cong c$ on $\mathcal M$, where we use the notation $\alpha \cong \beta$ to mean $\alpha = \gamma \beta$ for some unit $\gamma$.
 Now, another crucial property of the Segre mappings $v^k$ (see e.g.\ \cite{BER99b}) is that $(v^{k+1},\overline{v^k}) \in\mathcal M$ and $(v^{k-1},\overline{v^k}) \in \mathcal M$ for every $k$. Consequently, equation \eqref{eq16} implies that
 \begin{equation}\label{a}
 b\circ v^{k+1} \cong c\circ \overline{v^k}
,\quad
c \circ \overline{v^{k}} \cong b\circ v^{k-1}.
 \end{equation}
We deduce that $b\circ v^{k+1} \cong b\circ v^{k-1}$ for all $k\ge 2$. By induction, we obtain, for every positive integer $j$,
\begin{equation*}
b\circ v^{2j+1} \cong b\circ v^1.
\end{equation*}
Hence, since $b\circ v^1 \equiv 0$ by \eqref{eq16a}, we conclude that $b\circ v^{2j+1} \equiv 0$. Since, by Lemma~\ref{divisor}, $b(z,w)$ is a divisor of $(\det  H_Z(z,w))^d$, it follows that
$$\det  H_Z \circ v^{2j+1} \equiv 0.$$
This completes the proof of Lemma \ref{propagation}.
 \end{proof}

 We may now prove Theorem~\ref{cor}.

 \begin{proof}[Proof of Theorem $\ref{cor}$] Since $\jac H \not\equiv 0$ and $M$ is of finite type at $p$, it follows from Proposition 2.3 in \cite{ER06} that $M'$ is of finite type at $p'$. Also, by Theorem~\ref{main}, $H$ is CR transversal to $M'$ at $p$ and thus, in particular, transversally regular to $M'$ (see \cite{ER06}). Consequently, by Theorem 6.1 in \cite{ER06}, if $M$ is also essentially finite at $p$, then $H$ is finite mapping and $M'$ is essentially finite at $p'$. If, in addition,  $M$ is finitely nondegenerate at p, then Theorem 6.6 in \cite{ER06} asserts that $H$ is local biholomorphism near $p$.
 \end{proof}
Now, to prove Theorem~\ref{holomorphicnondegeneracycase}, we shall need the following proposition.
\begin{proposition}\label{ngocanh}
Let $M, M' \subset \C^N$ be smooth generic submanifold of the same dimension through $0$ and $H: (\C^N, 0) \to (\C^N,0)$ a germ of holomorphic mapping such that $H(M) \subset M'$. Assume that $M$ is holomorphically nondegenerate at $0$. If $H$ is CR transversal to $M'$ at $0$ then  $\jac H \not\equiv 0$.
\end{proposition}
\begin{proof} The idea of the proof is taken from \cite{A08}. Assume, in order to reach a contradiction, that $\jac H \equiv 0$. Then, there is a non trivial $N$-vector $U(z,w)$ with components in the field of fractions of $\C[[z,w]]$ such that
\begin{equation}\label{anhyeuem}
H_Z(z,w) U(z,w) \equiv 0.
\end{equation}
By multiplying \eqref{anhyeuem} with a suitable power series if necessary, we may assume that $U(z,w)$ has  components in $\C[[z,w]]$. Thus, we can consider the following nontrivial formal holomorphic vector field
\[
L = \sum_{j=1}^n U_j(z,w) \frac{\partial}{\partial z_j} + \sum_{l=1}^d U_{l+n}(z,w) \frac{\partial}{\partial w_l}.
\]
It follows from \eqref{anhyeuem} that $LF_j = LG_k = 0$ for all $j=1,\dots n, k=1,\dots d$.
Now, since $H$ sends $M$ into $M'$ we have
\begin{equation}\label{caigivay}
\bar G(\chi,\tau) - \bar Q'(\bar F(\chi,\tau),F(z,w), G(z,w)) = \bar a(\chi,\tau,z,w)(\tau - \bar Q(\chi,z,w)).
\end{equation}
Applying $L$ to the left hand side of \eqref{caigivay}, we obtain
\begin{align*}
-\sum_{j=1}^n\bar Q'_{z_j}(\bar F&(\chi,\tau),F(z,w), G(z,w))(LF_j)(z,w)\\
 &- \sum_{l=1}^d \bar Q'_{w_l}(\bar F(\chi,\tau),F(z,w), G(z,w))(LG_l)(z,w) \equiv 0.
\end{align*}
Consequently, we must also have $L\big(\bar a(\chi,\tau,z,w)(\tau - \bar Q(\chi,z,w))\big) \equiv 0$. In other words,
\begin{equation}\label{cuoicung}
(L\bar a)\bar \rho + \bar a (L\bar \rho) \equiv 0
\end{equation}
where $\bar \rho(z,w,\chi,\tau) = \tau - \bar Q(\chi,z,w)$. Since $H$ is CR transversal to $M'$ at $0$, we have $\det \bar a(0) \ne 0$ and hence $\bar a$ is invertible in $\C[[z,w,\chi,\tau]]$. We deduce from \eqref{cuoicung} that
\[
L\bar\rho = - (\bar a)^{-1} (L\bar a)  \bar \rho.
\]
It follows that $L$ is tangent to $M$. This is a contradiction since $M$ is holomorphically nondegenerate. The proof is complete.
\end{proof}

Theorem~\ref{holomorphicnondegeneracycase} is now a direct consequence of Theorem~\ref{main} and Proposition~\ref{ngocanh}.

\end{document}